\documentclass{article}
\usepackage{amsmath}
\usepackage{amssymb}
\usepackage[shortlabels]{enumitem}
\usepackage{amsthm, amsfonts}
\usepackage{framed,tabu,float,caption}
\usepackage{verbatim}
\usepackage[T1]{fontenc}
\usepackage{biblatex}
\usepackage[colorlinks = true, citecolor = blue, hypertexnames=false]{hyperref}
\usepackage{xcolor}
\usepackage{mathtools}
\hypersetup{
    colorlinks,
    linkcolor={red!70!black},
    citecolor={blue!70!black},
    urlcolor={blue!80!black}
}

\theoremstyle{plain}
\newtheorem{theorem}{Theorem}[section]

\newtheorem{lemma}[theorem]{Lemma}

\newtheorem{definition}[theorem]{Definition}
\newtheorem{fact}{Fact}
\newtheorem{claim}{Claim}
\newtheorem*{taskkey}{\hypertarget{task2}{Task 2'}}
\newtheorem*{Revextension}{\hypertarget{Extension}{Revisited Extension Lemma}}
\newtheorem*{Geo}{\hypertarget{GEO}{Revisited Geometrisation Lemma}}
\newtheorem*{Weak}{\hypertarget{CWeak}{Weaker Claim 1}}

\newtheorem{task}{Task}
\theoremstyle{definition}
\newtheorem*{definition*}{Definition}
\newtheorem*{remark}{Remark}

\newtheorem*{problem*}{Problem}
\newtheorem*{notation}{Notation}
\newtheorem*{theorem*}{Theorem}

\theoremstyle{remark}
\newtheorem{newclaim}{New Claim}

\DeclareMathOperator{\ad}{ad}
\DeclareMathOperator{\tr}{tr}
\DeclareMathOperator{\im}{im}

\DeclareMathOperator{\Aut}{Aut}
\DeclareMathOperator{\End}{End} 
\DeclareMathOperator{\Sym}{Sym} 
\DeclareMathOperator{\Alt}{Alt}
\DeclareMathOperator{\Supp}{Supp}
\DeclareMathOperator{\ustd}{std}
\DeclareMathOperator{\dc}{dc}
\DeclareMathOperator{\rstd}{\overline{std}}
\DeclareMathOperator{\sgn}{sgn}
\newcommand{\cU}{\mathcal{U}}
\DeclareMathOperator{\Ar}{Ar}
\DeclareMathOperator{\Ob}{Ob}
\newcommand{\Mod}{\mathbf{Mod}}
\newcommand{\generated}[1]{\left\langle#1\right\rangle}

\newcommand{\bN}{\mathbb{N}}
\newcommand{\bZ}{\mathbb{Z}}

\newcommand{\qoplus}{\mathbin{{\left(\mkern-2mu+\mkern-2mu\right)}}}

\begin{document}

\title{Minimal $\Alt$(6)-modules with an additive dimension in characteristic not 3}
\author{German Mata Gutiez \footnote{Sorbonne Université: german.mata\textunderscore gutiez@etu.sorbonne-universite.fr}}

\maketitle

\begin{abstract}
    This article takes place in the extension and generalization proposed by Corredor, Deloro and Wiscons in \cite{corredor2021} of the works of Dickson \cite{Dickson} and Wagner \cite{Wagner1976,Wagner1977} regarding the Classification of minimal $\Sym(n)$- and $\Alt(n)$-modules. We answer Remaining Task 2 from \cite{Chin_2026}.  In particular, we classify the minimal faithful $\Alt(6)$-modules with an additive dimension in characteristic not 3. We also highlight the remaining work  needed to complete the classification of minimal faithful $\Alt(n)$- and $\Sym(n)$-modules with an additive dimension.
\end{abstract}
\section{Introduction}
This work adds to the work in \cite{Chin_2026,corredor2021,MCENERNEY2025389} of generalizing and clarifying the results of Leonard Dickson \cite{Dickson} and Ascher Wagner \cite{Wagner1976,Wagner1977} regarding the classification of minimal $\Sym(n)$- and $\Alt(n)$-modules. As in \cite[p.~2]{corredor2021} "This identification is done on the minimal faithful representations of $\Sym(n)$ and $\Alt(n)$ on finite or infinite abelian groups in the presence of a rudimentary notion of dimension called 'additive dimension' but no a priori-and often no a posteriori-vector space structure". This lack of a vector space structure renders unusable Wagner's work through matrix decomposition in \cite[(4.4)~Lemma, p.~133] {Wagner1977} and \cite[Table~3.13, p.~112]{Wagner1976}. As stated, Wagner relies heavily on matrix algebra in \cite{Wagner1976,Wagner1977}. Our methods instead show that the representation theory does not involve any linear algebra and is `hard-coded' into the generators and relations. For instance, one may let $L$ be a Prüfer $p$-group in our Theorem, something certainly not covered by Wagner. Despite this, as in \cite{corredor2021} we stick to the strategy described in \cite{Wagner1977} of extending the $\Alt(6)$-action to a $\Sym(6)$-one. Nevertheless, in this setting, this is non-trivial due to the lack of higher structure. Furthermore, this level of generality does not allow character theory, not even basic tools such as Maschke's theorem.

The setting of modules with an additive dimension was introduced by Corredor, Deloro and Wiscons in \cite{corredor2021} in order to unify and extend various aspects of representation theory that support a notion of dimension arising from model theory. For example, it treats the classic finite dimensional vector spaces as well as classes of modules such as those with finite Morley Rank and those that are $o$-minimal. The methods are "elementary", based on generators and relations but yield strong and often clarifying results. 

Motivations and further complications can be found in \cite{corredor2021}. In the article, the authors classify $\Sym(n)$- and $\Alt(n)$-modules of least dimension for $n$ large enough, thus solving the large $n$ version of the following:
\begin{problem*}[{\cite[Problem]{corredor2021}}]Identify the minimal faithful $\Sym(n)$- and $\Alt(n)$-modules of least dimension for all $n$.
\end{problem*}
 Further work is continued by Chin, Deloro, Wiscons and Yu in \cite{Chin_2026} where they worked on the $\Alt(n)$-case for $n\le 6$, determining the minimal dimensions and classifying certain cases. In particular, they classified the $\Alt(6)$-modules of minimal dimension in characteristic 3. In addition, in \cite{MCENERNEY2025389} McEnerney and Wiscons through the notion of a 3D rotation module finished the classification of Alt(5)-modules of minimal dimension.
 
However, said articles leave open the following:
\begin{problem*}
    
Identify the minimal faithful $\Alt(n)$-modules of least dimension for $n=7,8,9$ in characteristic 2 and for $n=6$ in characteristic not 3.
\end{problem*}

 We solve the latter, identifying minimal $\Alt(6)$-modules when the characteristic is not $3$. We do so extending the action to a $\Sym(6)$-one. This approach is taken from \cite{corredor2021}. Nevertheless, despite not being relevant to the classification, this extension is achieved only in minimal dimension. The minimal dimensions, known from \cite[Lemma~3.11]{Chin_2026}, are 4 when the characteristic is 2 and 5 when the characteristic is different from 2 and 3. We do so extending the action to a $\Sym(6)$-one.
\subsection{The result}
We now state the main result. As in \cite{Chin_2026} and \cite{corredor2021}, $\Mod(G,d,q)$ denotes  the class of $G$-modules that carry an additive dimension, are dim-connected and of characteristic $q$. Precise definitions as well as the notion of dc-irreducibility are found later in the text in \S \ref{s.ModAddDim}. The result complements the following theorem proven in \cite{corredor2021}. We give the previous and actualized versions.
\begin{fact}[{\cite[Theorem]{corredor2021}}] Let $G = \Alt(n)$ or $\Sym(n)$. Suppose $V \in \Mod(G, d, q)$ is faithful and dc-irreducible with $d < n$. Assume $n \ge 7$; if $G = \Alt(n)$ and $q = 2$, further assume $n\ge 10$.
Then there is a dim-connected submodule $L \le V$ such that the structure of $V$ falls
into one of the following cases:
\begin{center}
\tabulinesep = 1.2mm
\begin{tabu} {X[0.8,c,m]|[1pt]X[0.3,c,m]|[1pt]X[2,c,m]}
$q$  & $d$   & Structure of $V$ \\ \tabucline[2pt]{-}
$q > 0$ and $q\mid n$ & $n-2$ & isomorphic to $\rstd(n, L)$ or $\sgn \otimes\,   \rstd(n, L)$ \\ \tabucline[1pt]{-}
$q > 0$ and $q\nmid n$ & $n-1$ & isomorphic to $\rstd(n, L)$ or $\sgn \otimes  \rstd(n, L)$ \\ \tabucline[1pt]{-}
$q = 0$ & $n-1$ & covered by $\ustd(n, L)$ or $\sgn \otimes \ustd(n, L)$ 
\end{tabu}
\end{center}
Moreover, when $q = 0$, the kernel of the covering map is 
$\langle{\sum_{i = 1}^{n-1} (e_i-e_n)}\rangle\otimes K$, in usual notation, for some $K \leq \Omega_n(L)$.
\end{fact}
We now add our result.
\begin{theorem*}\label{THEOREM}
    Let $G=\Alt(6)$ and denote by $\chi$ any outer automorphism of $\Alt(6)$. Suppose $V \in \Mod(G, d, q)$ is faithful and dc-irreducible with $d < 6$. Then there is a dim-connected submodule $L \le V$ such that the structure of $V$ falls into one of the following cases:
    \begin{center}
\tabulinesep = 1.2mm
\begin{tabu} {X[0.8,c,m]|[1pt]X[0.3,c,m]|[1pt]X[2,c,m]}
$q$  & $d$   & Structure of $V$ \\ \tabucline[2pt]{-}
$q=2$  & $4$ & isomorphic to $\rstd(6, L)$ or $\rstd(6,L)^\chi$\\ \tabucline[1pt]{-}
$q > 0$ and $q\nmid 6$ & $5$ & isomorphic to $\rstd(6, L)$ or $\rstd(6,L)^\chi$ \\ \tabucline[1pt]{-}
$q = 0$ & $5$ & covered by $\ustd(6, L)$ or $\ustd(6, L)^\chi$  
\end{tabu}
\end{center}
Here $\rstd(6,L)^\chi$ denotes the module obtained via the outer automorphism applied to $\Alt(6)$. Moreover, when $q = 0$, the kernel of the covering map is 
$\langle\sum_{i = 1}^{5} (e_i-e_6)\rangle\otimes K$, for some $K \leq \Omega_6(L)$.
\end{theorem*}

\subsection{Remaining work}

The main problem is not fully solved and some small $n$ cases remain. The work was divided in \cite{Chin_2026} into 4 different tasks; the first was solved in \cite{MCENERNEY2025389} and the second in the present article. We now state the remaining ones; for more details, refer to \cite[\S 1.3. Remaining work, Task 3,Task 4] {Chin_2026}.

\setcounter{task}{2}
\begin{task}
Classify the minimal $\Alt(n)$-modules for $n=7,8,9$ in characteristic 2.
\end{task}

\begin{task}
Classify the minimal $\Sym(n)$-modules for $n=5,6$ in characteristic not dividing $n$.
\end{task}

Furthermore, in this article, we revisited the Extension lemma (Fact~\ref{oldext}) to $n=6$ in minimal dimension and characteristic not 3. Is minimality required? (\cite[Theorem~2, (3)]{Chin_2026} shows that characteristic not $3$ is required). The lingering question is proposed as a new task.

\begin{taskkey}
    Can the \hyperlink{Extension}{Revisited Extension Lemma} be achieved in general dimension?
\end{taskkey}
\section{Background and preliminaries}

\subsection{Modules}\label{Modules}

Modules with an additive dimension were introduced in \cite{corredor2021}. We recall here some elementary tools taken from \cite{Chin_2026}.
\begin{notation}
Let $G$ be a group. Assume $V$ is a $G$-module and $g\in G$. We view $g$ as an element of $\Aut(V)$, the automorphism group of $V$;  $\End(V)$ denotes the endomorphism ring of $V$. When $g$ has finite order $m$, we define maps $\ad_g$ and $\tr_g$ in $\End(V)$ as follows:
\begin{itemize}
\item
$\ad_g := 1-g$;
\item
$\tr_g := 1 + g + \dots + g^{m-1}$.
\end{itemize}
We refer to these as the \textbf{adjoint} and \textbf{trace} maps. They commute.
If the module $V$ is clear from the context, we use notation $B_g = \ad_g(V)$ and $C_g = \tr_g(V)$ for the images of $V$ under these maps.

\end{notation}

\begin{fact}[{\cite[Lemma~2.1]{Chin_2026}}]\label{l.AdjointTrace}
Let $G$ be a group. Assume $V$ is a $G$-module and $g\in G$ has finite order. Then:
\begin{enumerate}
    \item\label{l.AdjointTrace.i:BTrivial} for $X\subseteq V$, $\ad_g(X) = 0$ if and only if $g$ centralizes $X$;
    \item\label{l.AdjointTrace.i:CCentral} $\ad_g(C_g) = 0$, so $C_g$ is centralized by $g$;
    \item\label{l.AdjointTrace.i:traceonB} $\tr_g(B_g) = 0$, so if $|g| = 2$, then $B_g$ is inverted by $g$;
    \item\label{l.AdjointTrace.i:BCIntersection} every element of $B_g \cap C_V(g)$ has order dividing $|g|$;
    \item\label{l.AdjointTrace.i:AdTrConjugation} for all $h\in G$, one has $h\circ \ad_g\circ h^{-1} = \ad_{hgh^{-1}}$ and $h\circ \tr_g\circ h^{-1} = \tr_{hgh^{-1}}$. 
\end{enumerate}
\end{fact}

\begin{fact}[{\cite[Lemma~2.2]{Chin_2026}}]\label{l.PermutingBandC}
Let $G$ be a group. Assume $V$ is a $G$-module and $g\in G$ has finite order. Then:
\begin{enumerate}
    \item\label{i:BCInversion} $B_g = B_{g^{-1}}$ and $C_g = C_{g^{-1}}$;
    \item\label{i:BCConjugation} for all $h\in G$, one has $h(B_g) = B_{hgh^{-1}}$ and $h(C_g) = C_{hgh^{-1}}$.
\end{enumerate}
In particular, if $h$ centralizes or inverts $g$, then $B_g$ and $C_g$  are $\langle h\rangle $-invariant.

\end{fact}

\begin{notation}
Let $G$ be a group acting on an abelian group $V$. We define \[[\underbrace{G,\ldots,G}_n,V] = \sum_{g_1,\ldots,g_n\in G} \ad_{g_1}\circ \cdots \circ \ad_{g_n} (V).\]
This agrees with the usual definition of the commutator  $[G,\ldots,G,V]$ (with appropriate bracketing) in the corresponding group $G\ltimes V$.
\end{notation}

\subsection{Modules with an additive dimension}\label{s.ModAddDim}
We now recall the framework of modules with an additive dimension. The definitions are taken directly from \cite{corredor2021}. Motivation and examples can be found in \cite{corredor2021}. For examples in low dimensions, refer to \cite{Chin_2026}. Throughout, we use $\ker f$ and $\im f$ for the kernel and image of a morphism $f$.

\begin{definition}
A \textbf{modular universe} $\cU$ is a collection of abelian groups $\Ob(\cU)$ and homomorphisms $\Ar(\cU)$ between them  that satisfy the following closure properties.
\begin{itemize}
\item{}[\textsc{inverses}]
If $f \in \Ar(\cU)$ is an isomorphism, then $f^{-1} \in \Ar(\cU)$.
\item{}[\textsc{products}] If $V_1, V_2 \in \Ob(\cU)$ and $f_1, f_2 \in \Ar(\cU)$, then 
$V_1 \times V_2 \in \Ob(\cU)$, and 
	$\Ar(\cU)$ contains $f_1 \times f_2$,  the projections $\pi_i : V_1\times V_2 \rightarrow V_i$, and the diagonal embeddings $\Delta_k:V_1\rightarrow V_1^k$.
\item{}[\textsc{sections}] 
If $W \leq V$ are in $\Ob(\cU)$, then $V/W\in \Ob(\cU)$ and $\Ar(\cU)$ contains the inclusion $\iota:W\rightarrow V$ and quotient $p:V\rightarrow V/W$ maps.
\item{}[\textsc{kernels/images}] 
If $f:V_1 \rightarrow V_2$ is in $\Ar(\cU)$, then $\ker f, \im f \in \Ob(\cU)$, and for all $W_1,W_2 \in \Ob(\cU)$,
	\begin{itemize} 
	\item
	if $W_1 \leq \ker f$, the induced map $\overline{f}\colon V_1/W_1 \to V_2$ is in $\Ar(\cU)$;
	\item 
	if $\im f \le W_2 \le V_2$,  the induced map $\check{f}\colon V_1 \to W_2$ is in $\Ar(\cU)$.
	\end{itemize} 
\item{}[$\bZ$-\textsc{module structure}]
If $V \in \Ob(\cU)$, then $\Ar(\cU)$ contains the sum  map $\sigma \colon V\times V \to V$ and the multiplication-by-$n$  maps $\mu_n\colon V \to V$.
\end{itemize}
The groups in $\Ob(\cU)$ are called the \textbf{modules} of $\cU$ and the homomorphisms in $\Ar(\cU)$ its \textbf{compatible morphisms}.

If $V$ is a module in  $\cU$ and a group $G$ acts on $V$ by \emph{compatible} morphisms of $\cU$, we say that $V$ is a \textbf{$G$-module} in $\cU$.
\end{definition}

\begin{definition}\label{d.ModAddDim}
Let $\cU$ be a modular universe.
An \textbf{additive dimension} on $\cU$ is a function $\dim\colon \Ob(\cU) \to \bN$ such that for all $f\colon V \to W$ in $\Ar(\cU)$, 
\[\dim V = \dim \ker f + \dim \im f.\]

If $(\cU, \dim)$ is a modular universe with an additive dimension and $V \in \Ob(\cU)$, we call $V$ \textbf{a module with an additive dimension}, leaving $\cU$ implicit from context. 
\end{definition}

\begin{definition}
A module $V$ with an additive dimension is called \textbf{dim-connected} (for dimension-connected) if every proper submodule $W<V$ in $\cU$ satisfies $\dim W < \dim V$.
\end{definition}

\begin{fact}[{\cite[Connectedness properties]{corredor2021}}]
Let $(\cU, \dim)$ be a modular universe with an additive dimension, and let $V_1,V_2$ be modules in $\cU$.
\begin{enumerate}
\item
If $V_1$ is dim-connected and $f\colon V_1 \to V_2$ is compatible, then $\im f$ is dim-connected.
\item
If $V_1$ and $V_2$ are dim-connected, then so is $V_1 \times V_2$.
\item
If $V_1, V_2 \leq V$ and $V_1$ and $V_2$ are dim-connected, then so is $V_1+V_2$.
\end{enumerate}
\end{fact}

\begin{notation}
Let $G$ be a group. We use $\Mod(G)$ to denote the class of all \emph{dim-connected} $G$-modules with an additive dimension (each coming from its own dimensional universe); $\Mod(G, d)$ denotes the subclass of those modules of dimension exactly $d$.
\end{notation}

\begin{definition}
If $A_1,\ldots,A_n$ are submodules of a module $V$, the sum $\sum A_i$ is said to be \textbf{quasi-direct} if $\dim \sum A_i = \sum \dim A_i$, in which case we write $\sum A_i = A_1\qoplus \cdots\qoplus A_n$.
\end{definition}


\begin{fact}[{\cite[Coprimality Lemma]{corredor2021}}]\label{fact.Coprimality}
Let $V\in \Mod(G)$ for a group $G$. Suppose $g\in G$ has prime order $p$ and $V$ is $p$-divisible. Then $V = B_g \qoplus C_g$.
\end{fact}

\begin{remark}\label{rem.dcComponent}
It follows from Fact~\ref{fact.Coprimality} that, under the same assumptions, $\dim C_V(g) = \dim C_g$; this is discussed in the Remarks following the proof of \cite[Coprimality Lemma]{corredor2021}. Since $\dim C_V(g)$ has a 
dim-connected submodule of equal dimension, it is unique, and we say that $C_g$ is the \textbf{dc-component} of $C_V(g)$, see \cite[Remarks, page 11]{corredor2021}. 
\begin{notation}
In general, for $B$ a submodule of $V$, denote by $B^{\dc}$ (if it exists) the dc-component of $B$.
\end{notation}
\end{remark}
\begin{notation}
For $G=\Sym(n)$ and $I \subset \{1,...,n\}$, we let $S_{I}\le \Sym(n)$ be the subgroup of permutations whose support is contained in $I$. For $\sigma \in G$ we let $S_\sigma= S_{\Supp(\sigma)}$ and $S_{\sigma^\perp}= S_{\Supp(\sigma)^c}$. For $G=\Alt(n)$, we define $A_{I}$,$A_\sigma$ and $A_{\sigma^\perp}$ similarly.
\end{notation} 
\begin{definition}\label{d.Characteristic}
Let $V$ be an abelian group. We say
\begin{itemize}
\item
 $V$ has \textbf{characteristic $p$}, for $p$ a prime, if $V$ is an elementary abelian $p$-group;
\item
$V$ has \textbf{characteristic $0$} if it is divisible, i.e.~if $V$ is $p$-divisible  for all primes $p$.
\end{itemize}
\end{definition}

\begin{notation}
Extending our earlier notation, we use $\Mod(G, d, p)$ to denote the subclass of $\Mod(G, d)$ consisting of those modules of characteristic $p$ (possibly equal to $0$).
\end{notation}

Not all modules have a well-defined characteristic, but it turns out that the ``irreducible'' ones do.

\begin{definition}\label{d.Irreducible}
Let $V\in \Mod(G)$ for some group $G$. We say that $V$ is \textbf{dc-irreducible} (for ``dimension-connected-irreducible''), if it has no non-trivial, proper, dim-connected $G$-submodule.
\end{definition}

\begin{fact}[{see \cite[Characteristic Lemma]{corredor2021}}]\label{fact.Characteristic}
Let $V\in \Mod(G)$ for some group $G$. If $V$ is dc-irreducible, then $V$ has a characteristic.
\end{fact}
\subsection{Known Results}
The following results are due to articles \cite{Chin_2026} and \cite{corredor2021}. Fact~\ref{CORR} is key for determining the dimension of certain submodules, namely the images of adjoint maps for $3$-cycles and bi-transpositions. Fact~\ref{MinDim} determines the dimension of minimal faithful $\Alt (6)$-modules. The Recognition lemma (Fact~\ref{RL}) finishes the identification of the modules once we have extended the action of $\Alt(6)$ to one of $\Sym(6)$. Fact~\ref{oldext} is adapted to the $\Alt(6)$-case in the \hyperlink{Exntension}{Revisited Extension Lemma}.

\begin{fact}[{\cite[Corollary 3.4.]{Chin_2026}}]\label{CORR}
Let $V \in \mathbf{Mod}(\Sym(3))$. If $\gamma \in \Sym(3)$ has order 3, then
any dim-connected, 1-dimensional $\Sym(3)$-submodule of $\ad_\gamma (V)$ has characteristic 3. In particular, if V is 3-divisible, then $\dim(\ad_\gamma(V))\ne 1$.

\end{fact}

\begin{fact}[{\cite[Lemma~3.11.]{Chin_2026}}]\label{MinDim}If $V\in \textbf{Mod}(\Alt(6),d)$ is faithful, then  $d\ge3$. Further, $d =3 $ is only possible when $V$ has characteristic 3, and $d=4$ is only possible when
$V$ has characteristic 2 or 3.
\end{fact}

\begin{fact}[{\cite[Recognition lemma]{corredor2021}}] \label{RL}
Let $n\ge 1$, $S := \Sym(n)$, and $V \in \Mod (S, d, q)$ be faithful and dc-irreducible. 
Suppose that for any transposition $\tau$, one has $[S'_{\tau^\perp}, B_\tau] = 0$.

Then for some abelian group $L$ and arrow $\varphi$ in $\cU$, there is a surjective morphism $\varphi\colon \ustd(n, L) \twoheadrightarrow V$ of $S$-modules.
Moreover:
\begin{itemize}
\item
if $0 \neq q \mid n$, then $\ker \varphi = C_{\ustd(n, L)}(S)$ and $V \simeq \rstd(n, L)$;
\item
if $0\neq q \nmid n$, then $\ker \varphi = 0$ and $V \simeq \rstd(n, L) \simeq \ustd(n, L)$;
\item
if $q = 0$, then $\ker \varphi = \generated{c}\otimes K$ where $c = \sum_{i = 1}^{n-1} f_i$ and $K \leq \Omega_n(L)$ is $0$-dimensional.
\end{itemize}
\end{fact}
\begin{fact}[{\cite[Extension lemma]{corredor2021}}]\label{oldext} Let $n \ge 7$, $A := \Alt(n)$, and $V\in\textbf{Mod}(A, d, q)$ be faithful and dc-irreducible. Suppose that for any bi-transposition $\alpha$, one has $[A_{\alpha^\perp}, B_\alpha] = 0$. Then:
\begin{itemize}
    \item  if $q = 2$ there is a unique compatible action of $\Sym(n)$ extending the $\Alt(n)$-structure;
    \item if $q \ne 2$ there are exactly two such, obtained from each other by tensoring with the
signature.
\end{itemize}
Moreover, up to tensoring with the signature, the extension satisfies that for any transposition $\tau \in S := \Sym(n)$, one has $[S_{\tau^\perp}, B_\tau ] = 0$.
\end{fact}

\section{The proof}

\begin{notation}
    
From now on, let $V$ be a dc-irreducible $\Alt(6)$-module of minimal dimension and characteristic $q$ (which is taken to be different from 3). Thanks to Fact \ref{MinDim} the dimension is either $4$ when $q=2$ or $5$ when $q\ne2,3$.
\end{notation}
We want to identify the minimal $\Alt(6)$-modules by extending the $\Alt(6)$-action to a $\Sym(6)$-one. The bulk of the proof is the \hyperlink{Extension}{Revisited Extension Lemma}, an $\Alt(6)$-modified version of Fact~\ref{oldext}.
The main modifications are as follows.
\begin{itemize}
    \item The restriction to modules of minimal dimension.
    \item A stronger geometric condition; for any 3-cycle $\gamma \in A$, $[A_{\gamma^\perp},B_\gamma]=0$ (proven in the \hyperlink{GEO}{Revisited Geometrisation Lemma}).
\end{itemize}  
The proof of the original lemma is based on six claims. Reading in detail  \cite[Extension lemma]{corredor2021} we remark that once the first claim is proven in the ($n=6$)-case the rest follows. Furthermore, extending the action allows us to conclude using the same reasoning as in \cite{corredor2021}. The \hyperlink{GEO}{Revisited Geometrisation Lemma} gives the necessary condition to apply the \hyperlink{Extension}{Revisited Extension Lemma}. After achieving the extension, the Recognition Lemma (Fact~ \ref{RL}) allows to conclude as it is still valid when $n=6$.

\subsection{Revisiting the Geometrisation lemma}
\begin{Geo} Let $V$ be a faithful dc-irreducible $\Alt(6)$-module of minimal dimension. Let $A=\Alt(6)$. Then,
    up to applying an outer automorphism, for every 3-cycle $\gamma $, one has $[A_{\gamma^{\perp}},B_\gamma]=0$.
\end{Geo}
The proof is based on modifying the argument of  \cite[Lemma~3.11]{Chin_2026}. The original proof gives a lower bound on the dimension of $A$-modules and is divided into several claims. We will adjust the reasoning from the second and fourth claims.
\begin{proof}We first modify \cite[Claim 2 p.~325]{Chin_2026} of the original article.

\setcounter{newclaim}{1}
\begin{newclaim}\label{Nclaim} Suppose $d\le 5$. If $\gamma,\delta \in A$ are disjoint 3-cycles then either:
    \begin{itemize}
        \item $B_\gamma=B_\delta$ and $C_\gamma=C_\delta$, or 
        \item $B_\gamma\le C_\delta$ and $B_\delta\le C_\gamma$.
    \end{itemize}
\end{newclaim}
    \begin{proof}
    It follows in the same way as \cite[Claim 2]{Chin_2026}.
    \end{proof}
      Denote by $\chi$ any outer automorphism of $\Alt(6)$ and $V'=V^\chi$. Remark that $V'$ is also a faithful dc-irreducible of minimal dimension. We now prove the condition $[A_{\gamma^{\perp}},B_\gamma]=0$ up to applying $\chi$. This amounts to ruling out the first case in New Claim \ref{Nclaim} in either $V$ or $V'$. Let $\gamma,\delta$ be two disjoint $3$-cycles. Also, notice $A_{\gamma ^\perp}=\langle \delta \rangle$. Suppose towards a contradiction $B_\gamma=B_\delta $ and $C_\gamma=C_\delta $. Notice the same holds in $V'$. As in the proof of  \cite[Claim~4,~p.~326]{Chin_2026} let $B:=B_\delta=B_\gamma$ and $\varepsilon := \gamma \delta $ the bi-$3$-cycle. Remark that if $\dim(B)\le 3$ or $B=B_\varepsilon$, the proof in \cite[Claim~4,~p.~326]{Chin_2026} still works, leading to a contradiction to faithfulness. Thus, $\dim(B) > 3$ and $B\ne B_\varepsilon$. We distinguish two cases.
    \begin{itemize}
        \item If $\dim(B)=\dim(V)$ then by dim-connectedness, $B=V$. After applying $\chi$, one has $B_\varepsilon'=V'$. But $B_\varepsilon' \le B_\delta' + B_\gamma' = B'$. This is a contradiction.
        \item Therefore, the remaining case is $\dim(B)=4$. Applying $\chi$, we can suppose $\dim(B_\varepsilon')=4$ and since $B_\varepsilon' \le B'$ we are either in the case $\dim(B')=5$ or $B_\varepsilon' = B'$. Both cases have been treated above. \qedhere 
    \end{itemize}
    
\end{proof}

\subsection{Revisiting the Extension lemma}
As stated, the lemma is analogous to the one in \cite[Extension lemma]{corredor2021}. Both extend the action of $\Alt(n)$ to one of $\Sym(n)$ thanks to the introduction of lines. The original proof is divided into several claims; the first gives the definition of lines and a geometric condition on them. Reading with care \cite{corredor2021}, once the first claim is proven in the case where $n=6$, the rest follows as for the other values. In order to prove Claim \ref{elemma:cl:L}, we calculate dimensions of key submodules and we go through an alternative definition of a line to conclude.
\begin{Revextension}\label{exten}
Let $V$ be a dc-irreducible $A:=\Alt(6)$-module of minimal dimension and characteristic not 3. Suppose that for any 3-cycle $\gamma$, one has $[A_{\gamma^\perp}, B_\gamma] = 0$.
Then:
\begin{itemize}
\item
if $q = 2$ there is a unique compatible action of $\Sym(6)$ extending the $\Alt(6)$-structure;
\item
if $q \neq 2,3$ there are exactly two such, obtained from each other by tensoring with the signature.
\end{itemize}
Moreover, up to tensoring with the signature, the extension satisfies that for any transposition $\tau \in S:=\Sym(6)$, one has $[S_{\tau^\perp}, B_\tau] = 0$.
\end{Revextension}
Like we said, only the proof of Claim~\ref{elemma:cl:L} of the original Extension Lemma (Fact~\ref{oldext}) needs to be adapted. We recall its statement.
\begin{claim}\label{elemma:cl:L}
Let $i \neq j$ be given. For distinct $a,b,k\notin \{i,j\}$, \[L_{(ij)}:= \im \ad_{(ij)(ab)} \circ \ad_{(ijk)} = (B_{(ij)(ab)} \cap B_{(ijk)})^{\dc}\] is nontrivial and independent of the choice of $a,b,k$. Also, $[A_{\{i, j\}^\perp},L_{(ij)}]=0$.
\end{claim}
Despite the similarities with the one in \cite{corredor2021}, Claim~\ref{elemma:cl:L} will take considerably longer to prove due to the smaller size of the group.

\subsubsection{Dimension work}
\begin{lemma}\label{DIMW}\leavevmode
\begin{enumerate}[(1)]
    \item\label{1:DIMW}  For any 3-cycle $\gamma \in \Alt(6)$, one has $ \dim(B_\gamma)=2$.
    \item\label{3:DIMW} Let $\delta$ be a $3$-cycle disjoint from $\gamma$ and $\lambda$ be a $3$-cycle such that $\Supp(\lambda)\ne \Supp(\gamma),\Supp(\delta)$.
    \item\label{2:DIMW} For any bi-transposition $\alpha \in \Alt(6)$, one has:
    \begin{itemize}
        \item if $q =2$, then $\dim(B_\alpha)=2$,
        \item if $q \ne 2,3$, then $\dim(B_\alpha)\in\{2,3\}$.
    \end{itemize}
    \begin{itemize}
        \item If $q=2$, one has $V=B_\gamma \qoplus B_\delta$.
        \item If $q\ne 2,3$, one has $V=(B_\gamma \qoplus B_\delta)+ B_\lambda$.
    \end{itemize}
    \item\label{4:DIMW} Let $i,j,k,\ell$ all distinct and $a,b\not \in \{i,j,k,\ell\}$. Then:
    \[\dim(\ad_{(ij)(k\ell)}(B_{(ija)})=\dim(\ad_{(ij)(k\ell)}(B_{(k\ell b)})\ne 0.\]
\end{enumerate}
    \end{lemma}
\begin{proof}
    We prove \ref{1:DIMW}. First, by faithfulness, $\dim B_\gamma \ne 0$. In addition, Fact~\ref{CORR} implies $\dim(B_\gamma)\ne 1$. Suppose towards a contradiction $\dim(B_\gamma) >2$. Since $q\ne 3$, Fact~\ref{fact.Coprimality} implies $V= B_\delta \qoplus C_\delta $. In addition, by the \hyperlink{GEO}{Revisited Geometrisation Lemma}, $B_\gamma= \ad_\gamma(V)= \ad_\gamma(C_\delta)$; but: 

 \[ \dim(C_\delta)=\dim(V)-\dim(B_\delta)\le 5-3=2 \]
 yields a contradiction. Hence, $\dim (B_\gamma)=2$.
    
    We prove \ref{3:DIMW} and \ref{2:DIMW}. Let $\alpha=(12)(45),\gamma=(123),\delta=(456)$ and $\lambda=(126)$. First, by faithfulness,  $\dim B_\alpha \ne 0$ and $\dim C_\alpha \ne 0$.  Remark that since $\alpha$ inverts $\gamma,\delta,\lambda $, both $\ad_\alpha$ and $\tr_\alpha$ normalize each of $B_\gamma,B_\delta,B_\lambda$. In particular, $\dim(\ad_\alpha(B_\delta) \cap \ad_\alpha(B_\gamma))=0$. Since $\dim(B_\gamma)=2$, the \hyperlink{GEO}{Revisited Geometrisation Lemma} implies $\dim(B_\gamma \qoplus B_\delta ) = 4$. Now Apply $(12)(36)$ and $(14)(25)$ to see:
    \[\dim(\ad_{\alpha}(B_\gamma))=\dim(\ad_{\alpha}(B_\lambda))=\dim(\ad_{\alpha}(B_\delta)).\]
    
    \begin{itemize}
        \item Suppose $q=2$. Then, dim-connectedness and equality of dimensions yields $V=B_\gamma \qoplus B_\delta$. We have achieved the ($q=2$)-case of \ref{3:DIMW}. Moreover, $B_\alpha=\ad_\alpha(B_\gamma) \qoplus \ad_\alpha(B_\delta)$. 
        
        We now prove \ref{2:DIMW}. If $\dim( B_\alpha) > 2$. One has $\dim(\ad_\alpha(B_\gamma))=2=\dim(\ad_\alpha(B_\delta))$. Therefore, $\dim(B_\alpha)=4$ and $V=B_\alpha$. But $\alpha$ is an involution, which yields:
        \[B_\alpha =\ad_\alpha(V)=\ad_\alpha(B_\alpha)=0,\] a contradiction to faithfulness. Therefore, $\dim(B_\alpha)\le2$ and $\dim(\ad_\alpha(B_\gamma))=1=\dim(\ad_\alpha(B_\delta))$. Thus, $\dim( B_\alpha)=2$. We have achieved the ($q=2$)-case of \ref{2:DIMW}.

        \item  Suppose $q\ne 2,3$. Note that the sum of all $B_\mu$ for $\mu$ a 3-cycle, is $\Alt(6)$-invariant and dim-connected; therefore, by dc-irreducibility it is $V$. Since $\dim(B_\gamma \qoplus B_\delta ) = 4$, one has $V=(B_\gamma \qoplus B_\delta) + B_\nu$ for some $\nu$ 3-cycle such that $B_\nu\not \le B_\gamma \qoplus B_\delta $. In particular, $\Supp(\nu)\ne \Supp(\gamma),\Supp(\delta)$. Remark that applying a suitable permutation, we can actually choose $\nu$ as long as it satisfies $\Supp(\nu)\ne \Supp(\gamma),\Supp(\delta)$. We have achieved the ($q\ne2,3$)-case of \ref{3:DIMW}.
        
         Choose $\nu=(126)=\lambda$. One has $B_\alpha=(\ad_\alpha(B_\gamma )\qoplus \ad_\alpha(B_\delta)) + \ad_\alpha(B_\lambda)$; and similarly for $\tr_\alpha$. Since $q\ne 2$, we see $\dim(B_\alpha)\in\{2,3\}$. We have achieved the ($q\ne2,3$)-case of \ref{2:DIMW}.
    \end{itemize}
    Finally, we prove \ref{4:DIMW}. Apply $(ikj\ell )(ab)$ to see:
    \[\dim(\ad_{(ij)(k\ell)}(B_{(ija)})=\dim(\ad_{(ij)(k\ell)}(B_{(k\ell b)}).\]
    In addition, showing $\dim(\ad_{(ij)(k \ell)}(B_{(ija)})\ne0$, is equivalent to showing that $(ij)(k\ell)$ does not centralize $(ija)$. Suppose it does. Conjugating, $(ij)(k\ell)$ centralizes $B_{ijt}$ for all $t \not \in \{i,j,k, \ell\}$ and also $B_{k \ell t}$ for all $t \not \in \{i,j,k,\ell\}$. Therefore, $C_V((ij)(k\ell))$ contains both $[A_{\{i,j\}^\perp} , V ]$ and $[A_{\{k,\ell\}^\perp} , V ]$, hence all of $[A,V]$. This contradicts our assumptions of dc-irreducibility and faithfulness. 
\end{proof}
\begin{remark}
   \ref{4:DIMW} is independent of the dimension of $V$. This result is in fact taken from \cite[Claim 1, p.~22] {corredor2021}. 
\end{remark}
\subsubsection{Alternative definition of a line}\label{lines}
We give an alternative definition of a line as the intersection of two planes. In fact, the lines used in the proof of the Extension Lemma (Fact~\ref{oldext}) are the dc-components of these new ones.
\begin{lemma}\label{Oldlines}
    Let $i\ne j$ be given. For distinct $k,\ell \not \in \{i,j\}$,
    \[ \mathcal{L}_{(ij)}:= B_{(ijk)}\cap B_{(ij\ell)}\]
    is not zero dimensional, does not depend on the choice of $k,\ell$ and $[A_{\{i, j\}^\perp},\mathcal{L}_{(ij)}]=0$.
    
    Furthermore, for $x,y,w,z$ all distinct:
    \[\dim (\mathcal{L}_{(xy)}\cap \mathcal{L}_{(wz)})=0.\]
\end{lemma}
\begin{proof}

    By conjugacy, we may take $i,j,k,\ell$ to be $1,2,3,6$.  We start by showing non-triviality; suppose $\dim(B_{(126)}\cap B_{(123)})=0$. Then, after applying $(13524)$, one has:  \[\dim(B_{(453)}\cap B_{(463)})=0.\] 
    By Lemma~\ref{DIMW}~\ref{1:DIMW}, one has $\dim B_{(123)}=\dim B_{(126)}=2$, therefore: 
    \[\dim(B_{(123)}\qoplus B_{(126)})=4.\]
    Since $\dim(B_{(453)})=2$ and $\dim(V)\le5$, one gets $\dim(\ker(\ad_{(456)}))\le 3$. Combined with the \hyperlink{GEO}{Revisited Geometrisation Lemma} and $\dim(B_{(123)}\qoplus B_{(126)})=4$, this yields: 
    \[\dim(\ad_{(456)}(B_{(126)}))= \dim(\ad_{(456)}(B_{(123)}\qoplus B_{(126)}))\ge 1.\]
    In addition, after applying $(12)(63)$, one has:
    \[\dim(\ad_{(453)}(B_{(123)}))\ge 1.\]
    Applying $(56)(12)$, one also has:
    \[\dim(\ad_{(463)}(B_{(123)}))\ge 1.\]
    In addition, since $(463)=(453)(465)$, one has:
    \begin{align*}
        \ad_{(463)}(B_{(123)})&=(1-(453)(465))(B_{(123)})\\
                   &= [(1-(465)) + (1-(453))(465)](B_{(123)})\\
                   &\le \ad_{(456)}(B_{(123)}) + \ad_{(453)}(B_{(123)})=\ad_{(453)}(B_{(123)}).
    \end{align*}
    Therefore:
    \[\dim(B_{(463)}\cap B_{(453)})\ge 1,\]
    which is a contradiction.
    
     Furthermore, see that $A_{\{1,2\}^\perp}= \langle (456),(453) \rangle$ so by the \hyperlink{GEO}{Revisited Geometrisation Lemma}, $[ A_{\{1,2\}^\perp},\mathcal{L}_{(12)}]=0$. Since $A_{\{1,2\}^\perp}$ is $2$-transitive, $\mathcal{L}_{(ij)}$ does not depend on the choice of $k,\ell$.
    
    Finally, take $x,y,w,z=1,2,3,4$. One has:
    \[\mathcal{L}_{(12)}\le B_{(123)} \text{ and } \mathcal{L}_{(45)}\le B_{(456)}.\]
    The \hyperlink{GEO}{Revisited Geometrisation Lemma} implies $\dim (\mathcal{L}_{(12)}\cap \mathcal{L}_{(45)}) \le \dim(B_{(123)}\cap B_{(456)})=0$.
\end{proof}
\subsubsection{Weaker claim 1}
We now show a weaker version of Claim \ref{elemma:cl:L} that will allow us to conclude. 
\begin{Weak}
    For $i,j,k,\ell,a,b$ all distinct:
    \[\dim (\ad_{(ij)(k\ell)}(B_{(ija)})\cap \ad_{(ij)(k\ell)}(B_{(ijb)}))\ne 0 .\]
\end{Weak}
\begin{proof}
    Take $i,j,k,\ell,a,b=1,2,4,5,3,6$. We reason by contradiction; suppose 
    \[\dim (\ad_{(ij)(k\ell)}(B_{(ija)})\cap \ad_{(ij)(k\ell)}(B_{(ijb)}))= 0.\]
    We start by setting up some facts towards the contradiction.
    \begin{enumerate}[label=(\roman*)]
        \item\label{1:Weak} $\dim \mathcal{L}_{(12)}=1$. We reason again by contradiction. Then, Lemma~\ref{Oldlines} shows $\dim\mathcal{L}_{(12)}=2$. Therefore, $B_{(123)}=B_{(126)}$. But Lemma~\ref{DIMW}~\ref{4:DIMW} proves $\dim (\ad_{(12)(45)}(B_{(126)}))\ne 0$. This is a contradiction.
        
        \item\label{2:Weak} $s:=\dim (\ad_{(12)(45)}(B_{(126)}))= 1$. We reason again by contradiction. Since $\ad_{(12)(45)}(B_{(126)}) \le B_{(126)}$, one has $s\le 2$. In addition, Lemma~\ref{DIMW}~\ref{4:DIMW} yields $s=2$. Applying $(45)(36)$ shows that:
        \[s=\dim(\ad_{(12)(45)}(B_{(126)}))=\dim(\ad_{(12)(45)}(B_{(123)})).\]  
        Furthermore, dim-connectedness and equality of dimensions imply:
            \begin{align*}
                \ad_{(12)(45)}(B_{(126)}) &=B_{(126)},\\
                \ad_{(12)(45)}(B_{(123)}) &=B_{(123)}.
            \end{align*}
            In addition, Lemma~\ref{Oldlines} yields:
            \[\dim (\ad_{(ij)(k\ell)}(B_{(ija)})\cap \ad_{(ij)(k\ell)}(B_{(ijb)}))=\dim (B_{(126)} \cap B_{(123)} )=\dim(\mathcal{L}_{(12)})\ne 0 .\]
            Hence, we have a contradiction.

        \item\label{3:Weak}$\dim(\mathcal{L}_{(12)} \cap \ad_{(12)(45)}(B_{(123)}) )= 0$. We reason again by contradiction. Then, since both terms of the intersection have dimension 1 (\ref{1:Weak}, \ref{2:Weak}) and $\ad_{(12)(45)}(B_{(123)})$ is dc, $\mathcal{L}_{(12)}^{\dc}=\ad_{(12)(45)}(B_{(123)})$. Apply $(45)(36)$ and write $\mathcal{L}_{(12)}=B_{(123)}\cap B_{(126)}$ to see that $\mathcal{L}_{(12)}^{\dc}=\ad_{(12)(45)}(B_{(126)})$. Therefore, combined with \ref{1:Weak}, one has:
        \[\dim (\ad_{(12)(45)}(B_{(126)})\cap \ad_{(12)(45)}(B_{(123)}))= \dim \mathcal{L}_{(12)}^{\dc}=1\ne 0.\]
        Hence, a contradiction arises. Therefore:
        \[\dim(\mathcal{L}_{(12)}\cap \ad_{(12)(45)}(B_{(123)}))=0.\]
         In particular, one has:
        \[ \dim(\ad_{(12)(45)}(\mathcal{L}_{(12)}))=\dim(\ad_{(12)(45)}(B_{(123)}\cap B_{(126)}))= 0.\]
    \end{enumerate}
    We are still after a contradiction; consider two cases.
    \begin{itemize}
        \item If $q=2$. Then, \ref{1:Weak}, \ref{2:Weak} and \ref{3:Weak} combined yield:
        \[B_{(123)}=\mathcal{L}_{(12)} \qoplus \ad_{(12)(45)}(B_{(123)}).\] 
         Furthermore, applying $\ad_{(12)(45)}$, which is an involution since $q=2$, one has $\ad_{(12)(45)}(B_{(123)})=0$. But by Lemma~\ref{DIMW}~\ref{4:DIMW}, $\dim (\ad_{(12)(45)}(B_{(123)}))\ne0$. Hence, we have a contradiction.

        \item Suppose now $q\ne 2,3$. We start by showing $\dim C_{(12)(45)}=2$. Item \ref{1:Weak} and Lemma~\ref{DIMW}~\ref{1:DIMW} imply  $\dim(B_{(123)} + B_{(126)})=3$. Therefore, Lemma~\ref{DIMW}~\ref{3:DIMW} shows:
        \[V=B_{(123)} + B_{(126)}+ B_{(456)}; \] 
        leading to:
        \[\dim((B_{(123)} + B_{(126)})\cap B_{(456)})=0.\] 
        Since:
        \[\dim (\ad_{(ij)(k\ell)}(B_{(ija)})\cap \ad_{(ij)(k\ell)}(B_{(ijb)}))= 0,\]
        one has:
        \[B_{(12)(45)}=\left( \ad _{(12)(45)}(B_{(123)}) \qoplus \ad_{(12)(45)}(B_{(126)})\right)\qoplus \ad_{(12)(45)}(B_{(456)}). \] 
        This, combined with Lemma~\ref{DIMW}~\ref{2:DIMW} and Fact~\ref{fact.Coprimality}, yields $\dim B_{(12)(45)}=3$ and $\dim C_{(12)(45)}=2$.
        
        Now, $(12)(45)$ inverts $(123)$ and $(126)$ and $q\ne 2$. Therefore, Fact~\ref{fact.Coprimality} implies:
        \[\mathcal{L}_{(12)}=\ad_{(12)(45)}(\mathcal{L}_{(12)}) \qoplus \tr _{(12)(45)}(\mathcal{L}_{(12)}).\]
        But \ref{3:Weak} yields $\dim (\ad_{(12)(45)}(\mathcal{L}_{(12)}))=0$. Thus, one has $\dim (\tr_{(12)(45)}(\mathcal{L}_{(12)}))=1$ and it is in fact the dc-component $\mathcal{L}_{(12)}^{\dc}$ of $\mathcal{L}_{(12)}$. Then, $2$-transitivity of $\Alt(6)$ implies for $x,y,w,z,h,j$ all distinct:
    \[\mathcal{L}_{(xy)}^{\dc}= \tr_{(xy)(wz)}(\mathcal{L}_{(xy)}).\]
    Furthermore, dim-connectedness and Lemma~\ref{Oldlines} imply:
    \[C_{(xy)(wz)}= \mathcal{L}_{(xy)}^{\dc} \qoplus \mathcal{L}_{(wz)}^{\dc} \text{ and } \mathcal{L}_{(hj)}^{\dc} \le C_{(xy)(wz)}.\]
    Since the sum of all $C_\beta$ where $\beta$ is a bi-transposition is $\Alt(6)$-invariant and dim-connected, it is $V$. Thus, the sum of all $\mathcal{L}_{(ij)}^{\dc} $ is $V$. Finally, 
    \[ \mathcal{L}_{(12)}^{\dc} +\mathcal{L}_{(45)}^{\dc} +\mathcal{L}_{(36)}^{\dc} +\mathcal{L}_{(24)}^{\dc} +\mathcal{L}_{(53)}^{\dc} \]
    has dimension at most 4. But it contains the sum of all $\mathcal{L}_{(ij)}^{\dc}$. This is a contradiction. \qedhere
        \end{itemize}
    \end{proof}
\subsubsection{Proof of Claim 1 and the Revisited Extension Lemma.}
\begin{proof}
    Let $\alpha=(ij)(ab),\gamma=(ijk),\lambda=(ij\ell)$. Lemma~\ref{DIMW}~\ref{1:DIMW} yields $\dim( B_\gamma)=2$. Combined with the \hyperlink{Wclaim}{Weaker Claim 1}, $\dim(\ad_\alpha(B_\gamma)\cap \ad_\alpha(B_\lambda)\in \{1,2\} $. Then, in both cases $\ad_\alpha(B_\gamma)=\ad_\alpha(B_\lambda)\le \mathcal{L}_{(ij)}$. The rest follows from the properties of the alternative lines.
\end{proof}
\subsection{Assembling the theorem}
\begin{proof}[Proof of the theorem]
Let $V$ be a faithful, dc-irreducible $\Alt(6)$-module of minimal dimension $d$ and characteristic $q\ne 3$. By Fact \ref{MinDim}, $d=2$ if $q=2$ and $d=5$ if $q\ne 2,3$. Apply in order the \hyperlink{GEO}{Revisited Geometrisation Lemma} to enforce the stronger geometric condition ($[A_{\gamma^\perp},B_\gamma]=0$ for every 3-cycle $\gamma$), the Extension lemma for $n=6$ (\hyperlink{Extension}{Revisited Extension Lemma}) to extend the action and the Recognition Lemma~(Fact \ref{RL}) to conclude the identification. 
\end{proof}
\section{Acknowledgments}
Thanks to my family for their education and support throughout my life. I would also like to thank Adrien Deloro for the multiple revisions and corrections of the work as well as his tutoring, Simon André for his tutoring, and Joshua Wiscons for his advice. 

This work was done entirely by humans.
\section{References}
\printbibliography[heading=none]
\end{document}